\documentclass[a4paper,12pt]{amsart}
\usepackage{amsmath,amssymb,amsfonts,amsthm} 
\usepackage[utf8]{inputenc}
\usepackage{tikz,relsize} 

\usepackage{pgf}
\usepackage{tikz-cd}

\usepackage{hyperref}
\usepackage[T1]{fontenc}

\usepackage{marginnote}
\usepackage{color}

\usepackage{enumitem}

\newcommand{\ses}[7]{\ensuremath{#1 \longrightarrow #2 \stackrel{#3}{\longrightarrow} #4 \stackrel{#5}{\longrightarrow} #6 \longrightarrow #7}}

\def\ns#1{\mathbb{#1}}
\def\N{\ns{N}}
\def\Z{\ns{Z}}
\def\Q{\ns{Q}}
\def\R{\ns{R}}
\def\C{\ns{C}}
\def\F{\ns{F}}
\usepackage[mathscr]{eucal}
\DeclareMathOperator{\spch}{\mathcal{S}}

\DeclareMathOperator{\Isom}{Isom}

\DeclareMathOperator{\GL}{GL}

\DeclareMathOperator{\Irr}{Irr}

\DeclareMathOperator{\Soc}{Soc}

\numberwithin{equation}{section}
\newtheorem{thm}{Theorem}[section]
\newtheorem{thm*}{Theorem}
\newtheorem{lm}[thm]{Lemma}
\newtheorem{cor}[thm]{Corollary}
\newtheorem{prop}[thm]{Proposition}

\newtheorem{rem}{Remark}[section]

\renewcommand{\Re}{\operatorname{Re}}
\renewcommand{\Im}{\operatorname{Im}}

\author{R. Lutowski}
\address{Institute of Mathematics, University of Gda\'{n}sk, Gda\'{n}sk, Poland}%
\email{rafal.lutowski@mat.ug.edu.pl}%
\title{Flat manifolds with homogeneous holonomy representation}

\date{}

\begin{document}

\maketitle

\begin{abstract}
We show that a rational holonomy representation of any flat manifold except torus must have at least two non-equivalent irreducible subrepresentations. As an application we show that if a K\"ahler flat manifold is not a torus then its holonomy representation is reducible.
\end{abstract}

\section{Introduction}

Let $n \in \N$. Let $\Gamma \subset \Isom(\R^n)$ be a \emph{crystallographic group}, i.e. a discrete and cocompact subgroup if isometries of the $n$-dimensional euclidean space $\R^n$. By the first Bieberbach theorem (see \cite[Theorem 2.1]{Sz12}, \cite[Theorems 3.1]{Ch86}) $\Gamma$ fits into the following short exact sequence
\begin{equation}
\label{eq:ses}
\ses{0}{M}{}{\Gamma}{\pi}{G}{1}
\end{equation}
where $G$ is a finite group and $M$ is a maximal abelian normal subgroup of $\Gamma$.
Moreover, $M$ is free abelian group of rank $n$ and it admits a structure of a faithful $G$-module, defined by the conjugations in $\Gamma$. The representation $\varphi \colon G \to \GL(M)$ which corresponds to this module is called the \emph{integral holonomy representation} of $\Gamma$. 
%faithful $G$-lattice of rank $n$, i.e. $M \cong \Z^n$ is a faithful $G$-module (see \cite[Theorem 2.1]{Sz12}). 
%The action of $G$ on $M$ is defined by so called \emph{integral holonomy representation} $\varphi \colon G \to \GL(M)$ of $\Gamma$ as follows:
%\[
%\varphi_g(m) = \gamma m \gamma^{-1}
%\]
%for every $g \in G, m \in M$ and $\gamma \in \Gamma$ such that $\pi(\gamma) = g$.

If $\Gamma$ is torsion-free we call it a \emph{Bieberbach group}. In this case the orbit space $X = \R^n/\Gamma$ is a closed Riemannian manifold with sectional curvature equal to zero, a \emph{flat manifold} for short. Moreover $\Gamma$, which is isomorphic to the fundamental group of $X$, determines the manifold up to affine equivalence. It is worth to notice that -- up to conjugation in $\GL(n,\R)$ -- $\varphi(G)$ is equal to the holonomy group of $X$. 

%A proper \emph{sublattice} of a $G$-lattice is a submodule of lower rank. \emph{Irreducibility} of a $G$-lattice means that its only $G$-submodule of lower rank is $0$. Let 
%\[
%M = M_k \supset M_{k-1} \supset \ldots \supset M_0 = 0
%\]
%be a composition series for $M$. 

Let $\F$ be a field. We will say that $G$-module $M$ is $\F$-\emph{homogeneous} if $\F G$-module
$\F \otimes_\Z M$ has only one homogeneous component, i.e. all of its irreducible submodules are isomorphic.
%all $\Q \otimes_\Z (M_i/M_{i-1})$ are isomorphic $\Q G$-modules, for $i=1,\ldots,k$. 
%In the case when $\Gamma$ is torsion-free and $M$ is homogeneous, we will call $X$ a \emph{homogeneous flat manifold}.
In the particular case $\F = \Q$ we will call $M$ \emph{homogeneous}. The same nomenclature applies to the corresponding representations of $G$.

In this paper we prove the following theorem:
\begin{thm*}
\label{thm:main}
The only Bieberbach groups with homogeneous integral holonomy representation are the fundamental groups of flat tori.
\end{thm*}

We use the theorem to prove that the holonomy representation of any K\"ahler flat manifold may not be $\C$-homogeneous and in particular -- irreducible.
Both of these results are generalization of the one presented in \cite{HS90} where the authors consider the case of $\Q \otimes_\Z M$ being irreducible.

The paper is organized as follows: in Section 2 we recall a notion of lattices and sublattices. Section 3 describes a reduction of the problem of homogeneous flat manifolds to the case where the holonomy group is simple. It is in fact a generalization of \cite[Section 2]{HS90} to the homogeneous case. The next two sections are devoted to the proof of Theorem \ref{thm:main} and its application to the class of K\"ahler flat manifolds.

\section{Lattices}

\begin{rem}
Since we consider tensor product over integers only, we will omit the subscript $\Z$ from now on.
\end{rem}

Let $G$ be a finite group. 
For the convenience of our further considerations we introduce a notion of a $G$-\emph{lattice}, i.e. a $G$-module which is a free $\Z$-module of finite rank.
%As mentioned in the introduction by a lattice $M$ is a $G$-module which is free abelian group of finite rank. 
A $G$-submodule $M'$ of $M$ is called a \emph{sublattice} if it is pure as $\Z$-submodule. As a consequence of this definition, we will call $M$ an \emph{irreducible} lattice if its only submodule of lower rank is $0$ or, equivalently, $\Q \otimes M$ is a simple $\Q G$-module. This approach allows us to build a descending chain
\[
M = M_k \supset M_{k-1} \supset \ldots \supset M_0 = 0
\]
of $G$-modules such that $M_{i-1}$ is a maximal sublattice of $M_i$ and hence $M_i/M_{i-1}$ is an irreducible $G$-lattice, for $i=1,\ldots,k$. Hence we can speak about composition series and composition factors for lattices (see \cite[\S 73]{CR62} for more detailed description).

\section{Reduction}

As in the previous section, let $G$ be a finite group.
For any $G$-lattice $M$ denote by $\Irr(G,M)$ the set of irreducible characters arising from constituents of $\C \otimes M$.  Note that if $M$ is a homogeneous $G$-lattice and $S$ is a composition factor of $M$ then $\Irr(G,M) = \Irr(G,S)$. We immediately get the following generalization of \cite[Lemma 2.1]{HS90}:

\begin{lm}
\label{lm:pb}
Let $M$ be a homogeneous $G$-lattice. Let $p$ be a prime and $\Z_p$ -- the ring of $p$-adic integers. Suppose that $\Z_p \otimes M$ contains an indecomposable direct summand in the principal $\Z_pG$-block. Then every irreducible constituent of $\C \otimes M$ is in the principal $p$-block of $G$.
\end{lm}

%Let $p$ be any prime. Denote by $\Z_p$ the ring of $p$-adic integers. Similarly as before, by $\Z_pG$-lattice we understand a $\Z_pG$-module, which is a free $\Z_p$-module of finite rank. We begin with an easy lemma.
%
%\begin{lm}
%Let $G$ be any finite group and $M$ -- a $G$-lattice which does not contain the trivial $G$-module. If $H^2(G,M) \neq 0$ then there exists a composition factor $S$ of $M$ such that $H^2(G,S) \neq 0$
%\end{lm}
%\begin{proof}
%The proof is done by induction on the number of composition factors of $M$. If $M$ is irreducible $G$-lattice then $S=M$, of course. Now let $L \subset M$ be an irreducible $G$-lattice. We have the following short exact sequence of $G$-lattices
%\[
%\sses{0}{L}{M}{Q}{0}
%\]
%which induces the following long exact sequence of cohomology groups
%\[
%\ldots \rightarrow H^1(G,M) \longrightarrow H^2(G,Q) \longrightarrow H^2(G,M) \longrightarrow H^2(G,L) \longrightarrow \ldots
%\]
%By assumption $H^1(G,M) \cong M^G = 0$ and if $H^2(G,L) = 0$ then $H^2(G,Q) \neq 0$. Since $Q$ has less composition factors than $M$, the result follows by induction.
%\end{proof}

Let $\Gamma$ be a crystallographic group which fits into the sequence \eqref{eq:ses} and let $\alpha \in H^2(G,M)$ be the cohomology class corresponding to this extension. Then $\Gamma$ is torsion-free if and only if $\alpha$ is \emph{special}, i.e. the restriction of $\alpha$ to every subgroup of $G$ of prime order is non-zero (see \cite[Theorem III.2.1]{Ch86} for example).

Note that in \cite[Lemma 2.2]{HS90} the irreducibility of a lattice was only needed to use \cite[Lemma 2.1]{HS90}. Hence we can generalize it to the homogeneous case:

\begin{lm}
\label{lm:se}
Let $M$ be a homogeneous $G$-lattice such that $H^2(G,M)$ contains a special element. Let $S$ denote a simple component in the socle $\Soc(G)$ (the product of minimal normal subgroups) of $G$. Then we have:
\begin{enumerate}[label={\normalfont(\alph{enumi})}]
\item \label{lm:hs:a} If $\vartheta \in \Irr(G,M)$, then $\vartheta$ is in the principal $p$-block for every prime $p$ dividing $|G|$.
\item \label{lm:hs:b} If $\psi \in \Irr(S,M)$, then $\psi$ is in the principal $p$-block for every prime $p$ dividing $|G|$.
\item \label{lm:hs:c} Let $p$ be a prime dividing $|S|$ such that a Sylow $p$-subgroup of $S$ is cyclic. Then there is $\Theta \in \Irr(S,M)$ which has the following position on the oriented Brauer tree:

{
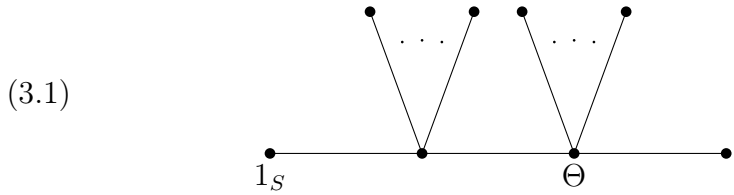
\begin{equation}
\label{eq:brauer}
\newlength{\rad}
\setlength{\rad}{2cm}
%\centering
\begin{tikzpicture}[baseline=(current  bounding  box.center)]
\fill (0,0) circle (2pt) node[below]{$1_S$};
\fill (\rad,0) circle (2pt);
\fill (\rad,0)++(110:\rad) circle (2pt);
\fill (\rad,0)++(100:.75\rad) circle (.5pt);
\fill (\rad,0)++( 90:.75\rad) circle (.5pt);
\fill (\rad,0)++( 80:.75\rad) circle (.5pt);
\fill (\rad,0)++(70:\rad) circle (2pt);
\draw (\rad,0) -- +(70:\rad);
\draw (\rad,0) -- +(110:\rad);
\fill (2\rad,0) circle (2pt) node[below] {$\Theta$};
\draw (2\rad,0) -- +(70:\rad);
\draw (2\rad,0) -- +(110:\rad);
\fill (2\rad,0)++(110:\rad) circle (2pt);
\fill (2\rad,0)++(100:.75\rad) circle (.5pt);
\fill (2\rad,0)++( 90:.75\rad) circle (.5pt);
\fill (2\rad,0)++( 80:.75\rad) circle (.5pt);
\fill (2\rad,0)++(70:\rad) circle (2pt);
\fill (3\rad,0) circle (2pt);
\draw (0,0) -- (3\rad,0);
\end{tikzpicture}
\end{equation}

}

\end{enumerate}
\end{lm}

Denote by $\spch(G)$ the set of those complex irreducible characters $\chi$ of $G$ such that for every prime $p$ dividing $|G|$:
\begin{enumerate}[label=(\arabic{enumi})]
\item $\chi$ is in the principal $p$-block;
\item If Sylow $p$-subgroup of $G$ is cyclic, then $\chi$ has the position given by the position of $\Theta$ on the Brauer tree \eqref{eq:brauer} of the principal $p$-block.
\end{enumerate}

The base finding of \cite{HS90}, giving its main result, may be stated as follows:
\begin{prop}[{\cite[Section 3]{HS90}}]
Let $S$ be a non-abelian finite simple group. Then $\spch(S) = \emptyset$.
\end{prop}

Since for every finite group its socle is a product of simple groups, we immediately get the following corollary.

\begin{cor}
\label{cor:socle}
Let $M$ be a homogeneous $G$-lattice such that $H^2(G,M)$ contains a special element. Then $\Soc(G)$ is a product of elementary abelian groups.
\end{cor}

\section{Proof of Theorem \ref{thm:main}}

%Please note that essentially we present the proof given in \cite{HS90}. The part which describes the case of $p$-groups is changed, although the ideas of Hiss and Szczepański could have been used here as well.
In the case of non $p$-groups of the proof we follow \cite[Section 4]{HS90}.
%, although the ideas of Hiss and Szczepański could have been used here as well.

Let $n \in \N$, $\Gamma \subset \Isom(\R^n)$ be a Bieberbach group which fits into the short exact sequence \eqref{eq:ses} and $X = \R^n/\Gamma$ be a homogeneous flat manifold, i.e. $M$ is a homogeneous $G$-lattice. Denote by $\chi$ its character.

Assume that $G$ is a non-trivial group. Then $\Soc(G)$ is non-trivial. Now take any prime $p$ dividing $|\Soc(G)|$ and a Sylow $p$-subgroup $N$ of $\Soc(G)$. By Corollary \ref{cor:socle} $N$ is characteristic in the socle, hence it is normal in $G$.

If $q \neq p$ is another prime dividing the order of $G$ then, by Lemma \ref{lm:se}, every $\vartheta \in \Irr(G,M)$ is in the principal $q$-block and so, by \cite[Lemma IV.4.12]{F82}, has $N \subset O_{q'}(G)$ in its kernel. But if $k$ is the composition length of $M$ then
\[
\chi = k \left( \sum_{\vartheta \in \Irr(G,M)} \vartheta \right)
\]
has $N$ in its kernel and $M$ is not a faithful $G$-module.

On the other hand, if $G$ is a $p$-group and $C$ is a cyclic and normal subgroup of $G$ then, by the Clifford's theorem (see \cite[Theorem 49.2]{CR62}), the restriction $\chi'$ of $\chi$ to $C$ equals
\[
\chi' = \varphi_1+\varphi_2+\ldots+\varphi_l
\]
where $\varphi_1,\ldots,\varphi_l$ are characters arising from irreducible components of the $\Q C$-module $\Q \otimes M$. Moreover, they are pairwise conjugated, i.e.
\begin{equation}
\label{eq:conjugated}
\forall_{1 \leq i,j \leq l} \exists_{g \in G} \forall_{c \in C} \varphi_i(c) = \varphi_j(gcg^{-1}).
\end{equation}
By \cite[Theorem 7.1]{ML95} if $H^2(C,M)\neq 0$ then there exists $1 \leq i \leq l$ s.t. $\varphi_i$ is the trivial character. But then the formula \eqref{eq:conjugated} shows that $\varphi_i$ is trivial for every $1 \leq i \leq l$. And again we get that $M$ is not a faithful $G$-module.

The above considerations show that $G$ must be the trivial group and hence $X$ is a flat $n$-dimensional torus.

\section{K\"ahler flat manifolds}

Let $\Gamma$ be a discrete, cocompact and torsion-free subgroup of $U(n) \ltimes \C^n$. Then $X = \C^n/\Gamma$ is a K\"ahler flat manifold with $\pi_1(X) \cong \Gamma$ (see \cite[Proposition 7.1]{Sz12}). If $\mathcal{B} \colon U(n) \ltimes \C^n \to O(2n) \ltimes \R^{2n}$ is a homomorphism given by
\[
\mathcal{B}(A,a) = 
\left(
\begin{bmatrix}
\Re(A) & -\Im(A)\\
\Im(A) &  \Re(A)
\end{bmatrix},
\begin{bmatrix}
\Re(a)\\
\Im(a)
\end{bmatrix}
\right)
\]
then $\mathcal{B}(\Gamma) \subset \Isom(\R^{2n})$ is a Bieberbach group. Moreover if $G$ is a holonomy group and $\varphi \colon G \to \GL(n,\C)$ is a holonomy representation of $X$ then the map $G \to \GL(2n,\R)$ defined by
\[
g \mapsto 
\begin{bmatrix}
\Re(\varphi(g)) & -\Im(\varphi(g))\\
\Im(\varphi(g)) & \Re(\varphi(g))
\end{bmatrix}
\]
is equivalent to the (integral) holonomy representation of the flat manifold $\R^{2n}/\mathcal{B}(\Gamma)$. We get

\begin{thm*}
Holonomy representation of a K\"ahler flat manifold, which is not a flat torus, contains at least two $\C$-homogeneous components. In particular it is reducible.
\end{thm*}

\begin{proof}
%If $\varphi$ is irreducible then integral holonomy representation of $\mathcal{B}(\Gamma)$ is either irreducible or decomposes over the rationals into two equivalent representations. i.e. is homogeneous.
If $\varphi$ is $\C$-homogeneous then the integral holonomy representation of $\mathcal{B}(\Gamma)$ is homogeneous.
\end{proof}

\section*{Acknowledgment}

The author would like to thank Marek Hałenda, Gerhard Hiss and Andrzej Szczepański for helpful discussions.

%\begin{rem}
%\label{rem:hs}
%In the above lemma irreducibility of the lattice $M$ is involved only in the proof of the point \ref{lm:hs:a}. Statement \ref{lm:hs:b} follows from \ref{lm:hs:a} and \label{lm:hs:b} is true because $H^2(G,M)$ contains a special element.
%\end{rem}
%
%
%\begin{lm}
%Lemma \ref{lm:hs} is true if $M$ is a homogeneous $G$-lattice.
%\end{lm}
%
%\begin{proof}
%By Remark \ref{rem:hs} it is enough to prove the statement \ref{lm:hs:a}. Take any prime $p$ dividing $|G|$ and any Sylow $p$-subgroup $U \subset G$. By  
%\end{proof}

%\bibliographystyle{plain}
%\bibliography{bibl}

\begin{thebibliography}{1}

\bibitem{Ch86}
L.S. Charlap.
\newblock {\em Bieberbach groups and flat manifolds}.
\newblock Universitext. Springer-Verlag, New York, 1986.

\bibitem{CR62}
C.W. Curtis and I.~Reiner.
\newblock {\em Representation theory of finite groups and associative
  algebras}.
\newblock Pure and Applied Mathematics, Vol. XI. Interscience Publishers, a
  division of John Wiley \& Sons, New York-London, 1962.

\bibitem{F82}
Walter Feit.
\newblock {\em The representation theory of finite groups}, volume~25 of {\em
  North-Holland Mathematical Library}.
\newblock North-Holland Publishing Co., Amsterdam-New York, 1982.

\bibitem{HS90}
Gerhard Hiss and Andrzej Szczepa\'nski.
\newblock On torsion free crystallographic groups.
\newblock {\em J. Pure Appl. Algebra}, 74(1):39--56, 1991.

\bibitem{ML95}
Saunders MacLane.
\newblock {\em Homology}.
\newblock Classics in Mathematics. Springer-Verlag Berlin Heidelberg, reprint
  of the 1975 edition, 1995.

\bibitem{Sz12}
A.~Szczepa{\'n}ski.
\newblock {\em Geometry of crystallographic groups}, volume~4 of {\em Algebra
  and Discrete Mathematics}.
\newblock World Scientific Publishing Co. Pte. Ltd., Hackensack, NJ, 2012.

\end{thebibliography}

\end{document}